\documentclass[a4paper,11pt]{amsart}
 \usepackage{amsmath,amsthm,amssymb,latexsym,enumerate, color}
 \usepackage[linkbordercolor={1 0 0}, citebordercolor={0 1 0}]{hyperref}

\numberwithin{equation}{section}

\begin{document}

\newtheorem{thm}{Theorem}[section]
\newtheorem{prop}[thm]{Proposition}
\newtheorem{lem}[thm]{Lemma}
\newtheorem{cor}[thm]{Corollary}
\newtheorem{rem}[thm]{Remark}
\newtheorem*{defn}{Definition}

\newcommand{\DD}{\mathbb{D}}
\newcommand{\NN}{\mathbb{N}}
\newcommand{\ZZ}{\mathbb{Z}}
\newcommand{\QQ}{\mathbb{Q}}
\newcommand{\RR}{\mathbb{R}}
\newcommand{\CC}{\mathbb{C}}
\renewcommand{\SS}{\mathbb{S}}

\renewcommand{\theequation}{\arabic{section}.\arabic{equation}}

\def\Xint#1{\mathchoice
{\XXint\displaystyle\textstyle{#1}}%
{\XXint\textstyle\scriptstyle{#1}}%
{\XXint\scriptstyle\scriptscriptstyle{#1}}%
{\XXint\scriptscriptstyle\scriptscriptstyle{#1}}%
\!\int}
\def\XXint#1#2#3{{\setbox0=\hbox{$#1{#2#3}{\int}$}
\vcenter{\hbox{$#2#3$}}\kern-.5\wd0}}
\def\ddashint{\Xint=}
\def\dashint{\Xint-}

\newcommand{\Erfc}{\mathop{\mathrm{Erfc}}}    
\newcommand{\supp}{\mathop{\mathrm{supp}}}    
\newcommand{\re}{\mathop{\mathrm{Re}}}   
\newcommand{\im}{\mathop{\mathrm{Im}}}   
\newcommand{\dist}{\mathop{\mathrm{dist}}}  
\newcommand{\link}{\mathop{\circ\kern-.35em -}}
\newcommand{\spn}{\mathop{\mathrm{span}}}   
\newcommand{\ind}{\mathop{\mathrm{ind}}}   
\newcommand{\rank}{\mathop{\mathrm{rank}}}   
\newcommand{\ol}{\overline}
\newcommand{\pa}{\partial}
\newcommand{\ul}{\underline}
\newcommand{\diam}{\mathrm{diam}}
\newcommand{\lan}{\langle}
\newcommand{\ran}{\rangle}
\newcommand{\tr}{\mathop{\mathrm{tr}}}
\newcommand{\diag}{\mathop{\mathrm{diag}}}
\newcommand{\dv}{\mathop{\mathrm{div}}}
\newcommand{\na}{\nabla}
\newcommand{\nr}{\Vert}

\newcommand{\al}{\alpha}
\newcommand{\be}{\beta}
\newcommand{\ga}{\gamma}  
\newcommand{\Ga}{\Gamma}
\newcommand{\de}{\delta}
\newcommand{\De}{\Delta}
\newcommand{\ve}{\varepsilon}
\newcommand{\fhi}{\varphi} 
\newcommand{\la}{\lambda}
\newcommand{\La}{\Lambda}    
\newcommand{\ka}{\kappa}
\newcommand{\si}{\sigma}
\newcommand{\Si}{\Sigma}
\newcommand{\te}{\theta}
\newcommand{\zi}{\zeta}
\newcommand{\om}{\omega}
\newcommand{\Om}{\Omega}

\newcommand{\cC}{\mathcal{C}}
\newcommand{\cG}{{\mathcal G}}
\newcommand{\cH}{{\mathcal H}}
\newcommand{\cI}{{\mathcal I}}
\newcommand{\cJ}{{\mathcal J}}
\newcommand{\cK}{{\mathcal K}}
\newcommand{\cL}{{\mathcal L}}
\newcommand{\cM}{\mathcal{M}}
\newcommand{\cN}{{\mathcal N}}
\newcommand{\cP}{\mathcal{P}}
\newcommand{\cR}{{\mathcal R}}
\newcommand{\cS}{{\mathcal S}}
\newcommand{\cT}{{\mathcal T}}
\newcommand{\cU}{{\mathcal U}}
\newcommand{\cX}{\mathcal{X}}

\title[Behaviour for game-theoretic $p$-caloric functions]{Short-time behaviour for \\ game-theoretic $p$-caloric functions}

\author{Diego Berti}
\address{Dipartimento di Matematica ed Informatica ``U.~Dini'',
Universit\` a di Firenze, viale Morgagni 67/A, 50134 Firenze, Italy.}
    \email{diego.berti@unifi.it}

\author{Rolando Magnanini} 
\address{Dipartimento di Matematica ed Informatica ``U.~Dini'',
Universit\` a di Firenze, viale Morgagni 67/A, 50134 Firenze, Italy.}
    \email{magnanin@math.unifi.it}
    \urladdr{http://web.math.unifi.it/users/magnanin}

\begin{abstract}
We consider the solution of $u_t-\De^G_p u=0$ in a (not necessarily bounded) domain, satisfying $u=0$ initially and $u=1$ on the boundary at all times. Here, $\De^G_p u$ is the {\it game-theoretic} or {\it normalized} $p$-laplacian. We derive new precise asymptotic formulas for short times, that generalize the work of S. R. S. Varadhan \cite{Va} for large deviations and that of the second author and S. Sakaguchi \cite{MS-AM} for the heat content of a ball touching the boundary. We also compute the short-time behavior of the $q$-mean of $u$ on such a ball. Applications to time-invariant level surfaces of $u$ are then derived.
\end{abstract}


\maketitle

\raggedbottom

\section{Introduction}

The {\it game-theoretic $p$-laplacian} is formally defined by
\begin{equation}
\label{G-laplace}
\De^G_p u=\frac1{p}\,|\na u|^{2-p} \dv\left\{ |\na u|^{p-2} \na u\right\}
\end{equation}
or
$$
\De^G_p u=\frac1{p}\,\left\{ \De u+(p-2)\frac{\lan\na^2 u \na u, \na u\ran}{|\na u|^2}\right\}
$$
and can be written as
$$
\De^G_p u=\frac1{p}\,|\na u|^{2-p} \De_p u,
$$
where
$$
\De_p u=\dv\left\{ |\na u|^{p-2} \na u\right\}
$$
is the classical $p$-laplacian.  We can suppose that $p\in(1,\infty]$, if we agree that for $\De_\infty^G$ we intend the limit of $\De_p^G$ as $p\to\infty$. Notice that $2 \De_2^G=\De$, the classical Laplace operator.
\par
The use of $\De_p^G$ finds applications in several fields. Certainly in game theory, in the study   of the  so-called {\it tug-of-war} games (for $p=\infty$, \cite{PSSW}) and their variants {\it with noise} (for $1<p<\infty$, \cite{PS}, \cite{MPR-SJMA}). The extremal case $p=1$ is related to the motion of hypersuperfaces by the \emph{mean curvature flow} (\cite{ES}). The $1$-homogeneity of $\De_p^G$ also makes the corresponding parabolic operator $\pa_t-\De^G_p$ scaling invariant, and
that is useful in the some techniques of \emph{image processing}, where the brightness of an initial image does not affect the evolution process. As shown in \cite{Do}, the choice of the parameter $p$ affects in which direction the brightness evolves.
\par
It is evident that, for $p\not=2$, $\De^G_p$ and $\De_p$ are both nonlinear. However, $\De_p^G$ is somewhat reminiscent of the lost linearity of the Laplace operator, since it is $1$-homogeneous, differently from $\De_p$, which is instead $(p-1)$-homogeneous. The nonlinearity of  $\De_p^G$ is indeed due to its non-additivity. Nevertheless, $\De_p^G$ acts additively if one of the relevant summands is constant and, more importantly, on functions of one variable and on radially symmetric functions. We shall see in the sequel that these last properties are decisive for the purposes of this paper.

\medskip

Let $\Om$ be a domain in $\RR^N$, $N\ge 2$, and let $\Ga$ be its boundary. In this paper, we shall consider properly generalized solutions of the following problem:
\begin{eqnarray}
&u_t=\De_p^G u \ &\mbox{ in } \ \Om\times(0,\infty), \label{G-heat} \\
&u=0  \ &\mbox{ on } \ \Om\times\{ 0\}, \label{initial} \\
&u=1  \ &\mbox{ on } \ \Ga\times(0,\infty). \label{boundary}
\end{eqnarray}
\par
We point out that equation \eqref{G-heat} is not variational and hence, since $\De^G_p$ is also not defined at spatial critical points of $u$, a convenient way to define
a solution of \eqref{G-heat} is in the {\it viscosity sense}. Besides the seminal work of Crandall, Ishii and Lions (\cite{CIL}), where one can find the foundations of that theory, 
 an appropriate list of references for the particular case of $\De_p^G$ includes \cite{AP}, \cite{BG-IUMJ}-\cite{BG-CPAA}, \cite{Do}, \cite{ES}, \cite{Ju}, \cite{JK}, \cite{JS}, \cite{MPR-SJMA}. The two conditions in \eqref{initial} and \eqref{boundary} should be intended in a classical sense, that is $u$ should be continuous on $\ol{\Om}\times[0,\infty)$ away from $\Ga\times\{0\}$. The existence of a solution of \eqref{G-heat}-\eqref{boundary} in this sense can be obtained by adapting arguments contained in\cite{BG-CPAA} and \cite{MPR-SJMA} (see Remark \ref{existence of solutions}).
\par
The concern of this paper is the study of the behaviour of the solution of \eqref{G-heat}-\eqref{boundary} when $t\to 0^+$.
When $p=2$, it is well-known that the asymptotics of the solution $u(x,t)$ is controlled by the distance of the point $x\in\Om$ from $\Ga$; in fact, by a formula first derived by Freidlin and Wentzell \cite{FW} by probabilistic methods and later adjusted by Varadhan (\cite{Va}) to an analytical framework, we know that
\begin{equation}
\label{varadhan}
\lim_{t\to 0^+} 2t \log u(x,t)=-d_\Ga(x)^2 \ \mbox{ uniformly for } \ x\in\ol{\Om},
\end{equation}
where
$$
d_\Ga(x)=\min_{y\in\Ga} |x-y| \ \mbox{ for } \ x\in\Om.
$$
Another proof of \eqref{varadhan}, based on arguments of the theory of viscosity solutions, is available in \cite{EI}.
\par
The second author and S. Sakaguchi used \eqref{varadhan} to prove another formula that links even more the short-time behaviour of the solution $u$ to the geometric features of the domain. Indeed, in \cite{MS-PRSE}, for a sufficiently smooth domain, it is showed that, if $B_R(x)\subset\Om$ is a ball centered at $x\in\Om$ and touching $\Ga$ at a unique point $y_x\in\Ga$ (hence $R=d_\Ga(x)=|x-y_x|$), then
\begin{equation}
\label{product-curvatures}
\lim_{t\to 0^+} t^{-\frac{N+1}{4}} \int_{B_R(x)} u(z,t)\,dz= 
\frac{c_N\,R^{\frac{N-1}{2}}}{\sqrt{\Pi_\Ga(y_x)}},
\end{equation}
where
\begin{equation}
\label{def-function-Pi}
\Pi_\Ga(y_x)=\prod\limits_{j=1}^{N-1}
\Bigl[1-R\,\ka_j(y_x)\Bigr].
\end{equation}
Here, $\ka_j(y_x), j=1, \dots, N-1$ are the pricipal curvatures of $\Ga$ at $y_x$, and $c_N$ is a constant only depending on the dimension. In \cite{MS-AM}, the integral in \eqref{product-curvatures} is referred to as the {\it heat content} of the ball $B_R(x)$.
\par
Generalizations of the two formulas \eqref{varadhan} and \eqref{product-curvatures} to nonlinear settings can be found in \cite{MS-PRSE}, \cite{Sa}, for the evolutionary $p$-Laplace equation, and in \cite{MS-AIHP}, for a class of non-degenerate fast diffusion equations. 
In order to obtain \eqref{varadhan}, one can merely assume that $\Om$, bounded or not, is such that $\Ga=\pa\left(\RR^N\setminus\ol{\Om}\right)$ (see \cite{Va}, for the linear case; \cite{MS-JDE2}, \cite{MS-MMAS}, for the various nonlinear cases;  \cite{MS-IUMJ}, \cite{MS-JDE}, \cite{MS-JDE2}, for the extension to unbounded domains). 
\par
The aim of this article is to extend \eqref{varadhan} and \eqref{product-curvatures} to the case $p\not=2$, that is when the problem \eqref{G-heat}-\eqref{boundary} is considered. 
In Subsections \ref{sec:first-order-general-domains} and \ref{sec:asymptotics-heat-content},  we shall in fact prove for $1<p\le\infty$ the following companions to formulas \eqref{varadhan} and 
\eqref{product-curvatures}:
\begin{equation}
\label{varadhan-p}
\lim_{t\to 0^+} 4t \log u(x,t)=-p'\,d_\Ga(x)^2, 
\end{equation}
uniformly on every compact subset of $\ol{\Om}$,\begin{equation}
\label{product-curvatures-p}
\lim_{t\to 0^+} t^{-\frac{N+1}{4}}\int_{B_R(x)}u(z,t)\,dz=\frac{c_N\,R^{\frac{N-1}{2}}}{(p')^\frac{N+1}{4}\,\sqrt{\Pi_\Ga(y_x)}}.
\end{equation}
Here $p'=p/(p-1)$ is, as usual, the conjugate exponent of $p$, and the constant $c_N$ can be explicitly computed (see Theorem \ref{th:asymptotics-heat-content-p}).
It is worth noticing that \eqref{varadhan-p} and \eqref{product-curvatures-p} differ from \eqref{varadhan} and \eqref{product-curvatures} only in the constants at the right-hand sides.
In Section \ref{sec:first-parabolic}, we shall also see how to extend the validity \eqref{varadhan-p} to the case in which the constant in \eqref{boundary} is
replaced by a function $h:\Ga\times(0,\infty)\to\RR$ bounded below and above by positive constants (see Corollary \ref{positive  boundary data}).
\par
The integral at the left-hand side of \eqref{product-curvatures-p} is just a multiple of the mean value of $u(\cdot,t)$ on $B_R(x)$.  This can be replaced by other significant statistical quantities related to $u$. One that seems to be particularly appropriate in this context is
the so-called {\it $q$-mean} $\mu_q^u(x,t)$ of $u(\cdot,t)$ on $B_R(x)$, that can be defined for $1<q\le\infty$ as 
\begin{multline}
\label{p-mean}
\mbox{the unique } \mu=\mu_q^u(x,t) \ \mbox{ such that } \\ 
\nr u(\cdot,t)-\mu\nr_{L^q(B_R(x))}\le\nr u(\cdot,t)-\la\nr_{L^q(B_R(x))} \ \mbox{ for all } \ \la\in\RR.
\end{multline}
We shall show the following scale invariant formula for $\mu_q^u(x,t)$:
\begin{equation}
\label{p-mean-asymptotics}
\lim_{t\to 0^+} \left(\frac{R^2}{t}\right)^{\frac{N+1}{4 (q-1)}} \mu_q^u(x,t)=\frac{c_{N,q}}{(p')^\frac{N+1}{4 (q-1)}\,\Pi_\Ga(y_x)^\frac1{2(q-1)}}
\end{equation}
(see Theorem \ref{th:asymptotics-p-mean} for the value of $c_{N,q}$).
\par
Formula \eqref{varadhan-p} thus extends to a game-theoretic realm the information given by  \eqref{varadhan}, that was originally conceived in the context of \emph{large deviations theory}, where it is related to the asymptotic probability that a sample path in a stochastic process remains in a domain for a given amount of time (see \cite{FW}, \cite{EI}). 
Besides this kind of application, \eqref{varadhan-p}  can be used to derive topological and geometric information about the solution of \eqref{G-heat}-\eqref{boundary}. For instance, it can help us to study the set of minimum points of $u$ that, due to \eqref{varadhan}, indeed approaches that of $d_\Ga$, when $t\to 0^+$ (as seen in \cite{MS-AN} or \cite{BMS} for the case $p=2$). 
\par
Another application is to the study of the so-called {\it stationary} or
{\it time-invariant} level surfaces of solutions of evolutionary partial differential equations. 
A surface $\Si\subset\Om$ (of co-dimension $1$) is said to be a {\it time-invariant level surface} for a solution $u$ of \eqref{G-heat}-\eqref{boundary}, if it is a level surface for $u(\cdot,t)$ for any time $t>0$, that is, if there is a function $a_\Si:(0,\infty)\to\RR$ such that
$$
u(x,t)=a_\Si(t) \ \mbox{ for any } \ (x,t)\in\Si\times(0,\infty).
$$
Thus, it is not difficult to infer that, under suitable sufficient assumptions on $\Si$, \eqref{varadhan} 
implies that a time-invariant surface $\Si$ is always {\it parallel} to $\Ga$. 
\par
Time-invariant level surfaces are known to enjoy remarkable symmetry properties (see \cite{MS-AM}-\cite{MS-MMAS}), \cite{MPeS}, \cite{MPrS}, \cite{CMS-JEMS}, \cite{MM-NA}).  For instance, in \cite{MS-AM}, \cite{MS-IUMJ} and \cite{MPeS}, \eqref{product-curvatures} was used to prove that stationary isothermic surfaces must be spherical, if $\Ga$ is compact, or planar or cylindrical, if $\Ga$ is not compact and satisfies certain global assumptions. In fact, it was shown that $\Si$ is time-invariant if and only if the heat content of any ball $B_r(x)\subset\Om$  satisfies a {\it balance law}, that is it does not depend on $x$ for $x\in\Si$ (but only on $r$ and $t$). By using this fact for $r=R$ and \eqref{product-curvatures}, one can easily infer that
\begin{equation}
\label{weingarten}
\Pi_\Ga=\mbox{constant on } \Ga
\end{equation}
--- that is, it turns out that $\Ga$ is a {\it Weingarten-type} surface. The above mentioned sufficient assumptions on $\Ga$ and \eqref{weingarten} then help to obtain the desired symmetry results. 
\par
When $p\not=2$, due to the lack of linearity, it is not known that, if the solution of \eqref{G-heat}-\eqref{boundary} admits a time-invariant level surface $\Si$, then \eqref{weingarten} holds --- nevertheless, we believe that it should hold true in all the cases $1<p\le\infty$.
For now, by using formulas \eqref{product-curvatures-p} and \eqref{p-mean-asymptotics}, we can only claim that, if for $1<q<\infty$ the $q$-mean $\mu_q^u(x,t)$ defined in \eqref{p-mean} does not depend on $x$ for $x\in\Si$, then \eqref{weingarten} holds and hence some symmetry results can be inferred (see Subsection \ref{sec:symmetry}).
\par
The proof of \eqref{varadhan-p} follows the tracks of Varadhan's argument and is presented in Section \ref{sec:first-parabolic}. In \cite{Va}, by taking advantage of the linearity of the heat equation, \eqref{varadhan} is obtained by proving the formula
\begin{equation*}
\label{elliptic}
\lim_{\ve\to0^+} \ve\,\log u^\ve(x)=-d_\Ga(x), \ x\in\ol{\Om},
\end{equation*}
for the modified Laplace transform of $u$,
$$
u^\ve(x)=\ve^{-2} \int_0^\infty u(x,t)\,e^{-t/\ve^2} dt, \ x\in\ol{\Om},
$$
and then by employing an inverse-Laplace-transform argument.
\par
When $p\not=2$, we bypass the lack of linearity in two steps. First, by the same Varadhan's argument, we prove \eqref{varadhan-p} when $\Om$
is a half-space and a ball --- here, we take advantage of the linearity $\De_G$ on one-dimensional and radially symmetric functions. Secondly, the validity of \eqref{varadhan-p} is obtained by employing as barriers the global solution of \eqref{G-heat} (see \cite{BG-IUMJ}), that generalizes the fundamental solution of case $p=2$, and the solution of \eqref{G-heat}-\eqref{boundary} in the ball. 
\par
The proofs of \eqref{product-curvatures-p} and \eqref{p-mean-asymptotics} are carried out in Section \ref{sec:second-parabolic} and are based on {\it new sharper} short-time asymptotic estimates. In fact, in Theorems \ref{th:pointwise} and \ref{th:uniform}, under the assumption that $\Ga$ is of class $C^2$, we show that
$$
4t \log u(x,t)+p'\,d_\Ga(x)^2=O(t\log t) \ \mbox{ as } \ t\to 0^+,
$$
uniformly for $x\in\ol{\Om}$.
We then construct lower and upper barriers for $u(x,t)$ that depend on $d_\Ga$ and $t$. Finally, for those barriers we obtain suitable versions of \eqref{product-curvatures-p} and \eqref{p-mean-asymptotics},  by applying a geometrical lemma (\cite[Lemma 2.1]{MS-PRSE}). To obtain \eqref{p-mean-asymptotics}, the monotonicity and continuity properties of the $q$-mean $\mu_q^u(x,t)$ with respect to $u$ (see \cite{IMW}) turn out to be very useful.

\section{First-order asymptotics}
\label{sec:first-parabolic}

In this section, we shall prove \eqref{varadhan-p}. We start with an asymptotic formula for
the solutions of a family of elliptic problems related to \eqref{G-heat}-\eqref{boundary}. 

\subsection{An auxiliary elliptic problem}
\label{sec:elliptic}

For $\ve>0$, we consider the following auxiliary elliptic problem:
\begin{eqnarray}
&u-\ve^2\De^G_pu=0 \ &\mbox{ in } \ \Om, \label{G-elliptic} \\
&u=1  \ &\mbox{ on } \ \Ga. \label{elliptic-boundary}
\end{eqnarray}
This problem is somewhat related to \eqref{G-heat}-\eqref{boundary} in the sense that it is easy to show that, when $p=2$, due to the linearity of the Laplace operator, the function defined by
\begin{equation}
\label{u-epsilon}
u^\ve(x)=\frac1{\ve^2}\int_{0}^{\infty}u(x,\tau)\,e^{-\tau/\ve^2}\, d\tau, \ \mbox{ for } \ x\in\ol{\Om},
\end{equation}
satisfies \eqref{G-elliptic}-\eqref{elliptic-boundary}, if $u=u(x,t)$ is a solution of \eqref{G-heat}-\eqref{boundary} (see \cite{MS-AM}). When $p\not=2$, $\De_p^G$ is no longer linear in general. Nevertheless, it is linear on radially symmetric functions and this fact will be useful in the sequel. We shall denote by $B_R$ the ball with radius $R$ centered at the origin.

\begin{lem}
\label{solution in the ball}
Set $1<p\le\infty$ and $\Om=B_R$. Then, the (viscosity) solution of \eqref{G-elliptic} and \eqref{elliptic-boundary} is given by
\begin{equation} 
\label{ball solution formula}
u^\ve(x)=\begin{cases}
\displaystyle 
\frac{\int_{0}^{\pi}e^{\sqrt{p'}\frac{|x|}{\ve}\cos \te}(\sin\te) ^{\al}\,d\te}{\int_{0}^{\pi}e^{\sqrt{p'}\frac{R}{\ve}\cos\te}(\sin\te) ^{\al}\,d\te} \ &\mbox{ if } \ 1<p<\infty,
\vspace{7pt} \\
\displaystyle
\frac{\cosh(|x|/\ve)}{\cosh(R/\ve)}  \ &\mbox{ if } \ p=\infty,
\end{cases} 
\end{equation}
for $x\in\ol{B}_R$;  here $\al=\frac{N-p}{p-1}$.
\end{lem}

\begin{proof}
We observe that the function $u^\ve$ is twice-differentiable in $B_R$ and $x=0$ is its unique critical point. By a direct computation,  $u^\ve$ satisfies \eqref{G-elliptic}, away from its critical point. Lemma \ref{regular viscosity solutions} ensures that $u^\ve$ is a (viscosity) solution of \eqref{G-elliptic} in $B_R$. It is evident that $u^\ve(x)$ also satisfies \eqref{elliptic-boundary}. Thus, by the uniqueness of viscosity solutions of problem \eqref{G-elliptic}-\eqref{elliptic-boundary} (see Lemma \ref{elliptic cp}), we get our claim.
\end{proof}

\begin{thm}[Asimptotics as $\ve\to 0^+$]
\label{ell conv ball}
Set $1<p\le\infty$, $\Om=B_R$, and let $u^\varepsilon$ be the solution of \eqref{G-elliptic}-\eqref{elliptic-boundary} . 
\par
Then, for every $x\in \ol{B}_R$, it holds that
\[
\lim_{\ve\to 0^+}\ve\log\left[ u^\ve(x)\right]=-\sqrt{p'}\,d_\Ga(x),
\]
and the limit is uniform on $\ol{B}_R$.
\end{thm}
\begin{proof}
We present the proof of the more difficult case $1<p<\infty$. We observe that \eqref{ball solution formula} implies that
	\[
	u^\ve(x)\,e^{\sqrt{p'} \frac{R-|x|}{\ve}}=\frac{\int_{0}^{\pi}e^{-\sqrt{p'} \frac{|x|}{\ve} (1-\cos\te)}(\sin\te) ^{\al}\,d\te}{\int_{0}^{\pi} e^{-\sqrt{p'} \frac{R}{\ve} (1-\cos\te)}(\sin\te) ^{\al}\,d\te},
	\]
and hence
\begin{equation}
\label{ineq-log}
	\ve\,\log u^\ve(x)+\sqrt{p'}\,(R-|x|)=\log\left[\frac{\int_{0}^{\pi}e^{-\sqrt{p'} \frac{|x|}{\ve} (1-\cos\te)}(\sin\te) ^{\al}\,d\te}{\int_{0}^{\pi} e^{-\sqrt{p'} \frac{R}{\ve} (1-\cos\te)}(\sin\te) ^{\al}\,d\te}\right]^\ve.
\end{equation}
\par
Now,
each value of the function
$$
f_\ve(x)=\left\{\frac1{\int_{0}^{\pi}(\sin\te)^\al\,d\te}\int_{0}^{\pi}e^{-\sqrt{p'} \frac{|x|}{\ve},  (1-\cos\te)}(\sin\te) ^{\al}\,d\te\right\}^\ve
$$
can be viewed as the norm of the function (of $\te$) $e^{-\sqrt{p'} |x| (1-\cos\te)}$ in the Lebesgue space $L^\frac1{\ve}(0,\pi)$ equipped with the measure $(\sin\te)^\al d\te/\int_0^\pi (\sin\te)^\al d\te$, that is unitary. It is well known that such a norm increases as $\ve\to 0^+$ and tends to the function
$$
f_0(x)=\max\left\{e^{-\sqrt{p'} |x| (1-\cos\te)}: \te\in[0,\pi]\right\}=1, \ x\in\ol{\Om},
$$
as seen, for example, in \cite{LL}. 
\par
Therefore, from \eqref{ineq-log} we conclude that $\ve\log u^\ve+\sqrt{p'}\,d_\Ga$ converges pointwise to zero as $\ve\to 0^+$, since $d_\Ga(x)=R-|x|$. Moreover, the convergence is uniform on $\ol{B}_R$ owing to {\it Dini's monotone convergence theorem} (see \cite[Theorem 7.13]{Ru}). Indeed, $\ol{B}_R$ is compact, the sequence $\{ f_\ve\}_{\ve>0}$ is monotonic, and the functions $f_\ve$ and $f_0$ are continuous on  $\ol{B}_R$.
\end{proof}	

\subsection{The parabolic one-dimensional and radially symmetric cases}

Preliminarily, we focus our attention on the case of the half-space $H\subset\RR^N$ in which $x_1>0$. In this case the problem \eqref{G-heat}-\eqref{boundary} has a regular one-dimensional solution, as the following proposition states. In what follows, we will use the {\it complementary error function} defined by 
$$
\Erfc(\si)=\frac{2}{\sqrt{\pi}}\,\int_\si^\infty e^{-\tau^2} d\tau, \quad \si\in\RR.
$$

\begin{prop}
\label{Asymp hspace}
Set $\Om=H$; the solution of the problem \eqref{G-heat}-\eqref{boundary} is the function $\Psi$ defined by
$$
\Psi(x,t)=\sqrt{\frac{p'}{4\pi}}\int_{\frac{x_1}{\sqrt{t}}}^{\infty} e^{-\frac14 p' \si^2}\,d\si=
\Erfc\left(\frac{\sqrt{p'} x_1}{2\sqrt{t}}\right) \ \mbox{ for } \ (x,t)\in\ol{H}\times (0,\infty).
$$
\par
Moreover, it holds that 
$$
\lim_{t\to 0^+}4t\log \left[\Psi(x,t)\right]=-p'\,d_\Ga(x)^2, 
$$
uniformly for $x$ in every strip $\{ x\in\RR^N: 0\le x_1\le \de\}$ with $\de>0$.
\end{prop}

\begin{proof}
By direct calculation, we see that $\Psi$ is a classical solution of \eqref{G-heat}; an inspection tells us that it also satisfies conditions \eqref{initial} and \eqref{boundary}. 
\par
By a change of variables, we get that
$$
\Psi(x,t)=e^{-\frac{p' x_1^2}{4t}}\,\sqrt{\frac{p'}{4\pi}}\int_0^\infty e^{-\frac12 p' \frac{x_1}{\sqrt{t}}\,\si-\frac14 p' \si^2}d\si,
$$
and hence
$$
4t \log\Psi(x,t)+p' x_1^2= 4t\,\log\left( \sqrt{\frac{p'}{4\pi}}\int_0^\infty e^{-\frac12 p' \frac{x_1}{\sqrt{t}}\,\si-\frac14 p' \si^2}d\si \right).
$$
\par
Thus, for $0\le x_1\le\de$, we have that
$$
4t\,\log\left( \sqrt{\frac{p'}{4\pi}}\int_0^\infty e^{-\frac12 p' \frac{\de}{\sqrt{t}}\,\si-\frac14 p' \si^2}d\si \right)\le 4t \log\Psi(x,t)+p' x_1^2\le 0,
$$ 
and this implies the desired uniform convergence.
\end{proof}

To determine the short-time asymptotic behavior of the solution of \eqref{G-heat}-\eqref{boundary} in a ball, the following lemma, that holds for a general domain, will be useful.

\begin{lem}
\label{increasing in time}
Let $\Om$ be a domain in $\RR^N$ and $u(x,t)$ be the solution of problem \eqref{G-heat}--\eqref{boundary}. 
Then, for each $x\in \ol{\Om}$, the function
$
(0,\infty)\ni t\mapsto u(x,t)
$
is increasing.
\end{lem}

\begin{proof}
Fix $\tau>0$ and define the function $v=v(x,t)$ by
\[
v(x,t)=u(x,t+\tau), \ (x,t)\in \Om\times (0,\infty).
\]
Since \eqref{G-heat} is invariant under translations in time, one can easily check that $v$ satisfies \eqref{G-heat} and \eqref{boundary}, and is such that $v(x,0)=u(x,\tau)$, that is positive by the maximum principle for $u$.
Thus, the comparison principle (see Theorem \ref{parabolic cp})  gives that
$v\geq u$ on $\ol\Om\times(0,\infty)$, which implies the desired monotonicity for $u$.
\end{proof}

Next, we need to be sure that the viscosity solution of \eqref{G-heat}-\eqref{boundary} in a ball $B_R$ can be transformed into the viscosity solution of \eqref{G-elliptic}-\eqref{elliptic-boundary}.

\begin{lem}[Solution in a ball]
\label{radial laplace transform}
Set $1<p\le\infty$ and let $\Om$ be the ball $B_R$ centered at the origin with radius $R$; the function $u^\ve$, defined by \eqref{ball solution formula}, results from \eqref{u-epsilon}, where $u(x,t)$ is the (viscosity) solution of \eqref{G-heat}-\eqref{boundary}.
\end{lem}
\begin{proof}
Owing to the linearity of \eqref{G-heat}, the unique solution of $u(x,t)$ can be explicitly computed as a series expansion (see \cite{KKK}).
\par
In fact, if $1< p <\infty$, we have that
$$
u(x,t)=1-2\,\sum_{n=1}^{\infty}\frac{e^{-\frac{\ga_n^2}{p' R^2}\, t}}{\ga_nJ_{\be+1}(\ga_n)}\left(\frac{R}{|x|}\right)^{\be}J_\be\left(\frac{\ga_n}{R}|x|\right),
$$
where 
$J_\be$ is the Bessel function of the first kind of order $\be=(N-p)/(2p-2)$ and $\ga_n$ are the positive zeros of $J_\be$;  if $p=\infty$, we get instead:
$$
u(x,t)=1-\frac4{\pi}\sum_{n=1}^{\infty}\frac{(-1)^{n-1}}{2n-1}e^{-\frac{(2n-1)^2 \pi^2}{4R^2}\,t}\cos\left( (2n-1)\frac{\pi |x|}{2R}\right).
$$ 
\par
Tedious but straightforward calculations, based on formulas from \cite{AS}, show that $u(x,t)$
is transformed, by means of \eqref{u-epsilon}, into the function $u^\ve$, defined by \eqref{ball solution formula}.
\end{proof}

\begin{thm}[Short-time asymptotics in a ball]
\label{Asymptotic in the ball}
Set $1<p\le\infty$, $\Om=B_R$, and let $u(x,t)$ be the viscosity solution of \eqref{G-heat}-\eqref{boundary}. 
\par
Then, it holds that
\begin{equation}
\label{limit}
\lim_{t\to 0^+}4t\log u(x,t) =-p'\,d_\Ga(x)^2 \ \mbox{ for } \ x\in\ol\Om.
\end{equation}
\end{thm}

\begin{proof}
Given $x\in\Om$ there exists $y\in\Ga$ such that $|x-y|=d_\Ga(x)$ ($y$ is unique unless $x=0$). Let $H$ be the half-space containing $\Om$ and such that $\pa H\cap\Ga=\{ y\}$; notice that $d_\Ga(x)=d_{\pa H}(x)$.
\par 
Let $\Psi^y$ be the solution of \eqref{G-heat}-\eqref{boundary} in $H\times(0,\infty)$; since $\Om$ is contained in $H$, $\Psi^y$ obviously satisfies \eqref{G-heat} and \eqref{initial} for $\Om$ and, also, $\Psi^y\le 1$ on $\Ga\times(0,\infty)$.
By comparison (Theorem \ref{parabolic cp}), we get that $u\ge\Psi^y$ and hence
\begin{equation}
\label{half-space inequality}
4t\log u(x,t)\geq 4t\log\Psi^y(x,t) \ \mbox{ for } \ (x,t)\in \ol\Om\times(0,\infty).
\end{equation}
Thus, Theorem \ref{Asymp hspace} implies that
\begin{equation}
\label{liminf ball}
\liminf_{t\to 0^+} 4t\log u(x,t)\ge -p'\,d_{\pa H}(x)^2=
-p'\,d_\Ga(x)^2.
\end{equation}
The last limit is uniform on $\ol\Om$, as Proposition \ref{Asymp hspace} (with the choice $\de=2R$) informs us.
\par
Now, by  Lemma \ref{radial laplace transform}, for every $\ve>0$ the function
$u^\ve$ defined in \eqref{u-epsilon} is the solution of \eqref{G-elliptic}-\eqref{elliptic-boundary} in $\Om$. Thus, by Lemma \ref{increasing in time}, we have that
\[
\ve^2 u(x,t)\,e^{-t/\ve^2}\le \int_t^\infty u(x,\tau)\,e^{-\tau/\ve^2}\,d\tau \leq \int_{0}^{\infty}u(x,\tau)e^{-\tau/\ve^2}\,d\tau = \ve^2 u^\ve(x),
\]
and hence
$$
u(x,t)\le u^\ve(x)\,e^{t/\ve^2};
$$
the last inequality holds for any $t, \ve>0$. Next, we choose
$\ve=\la\, t$ and obtain that
$$
u(x,t)\le u^{\la t}(x)\,e^{1/\la^{2} t} \ \mbox{ for any } \ t>0.
$$
Therefore, 
\begin{equation}
\label{lambda-inequality}
4t\,\log u(x,t)\le \frac4{\la}\,(\la\,t)\,\log u^{\la t}(x)+\frac4{\la^2},
\end{equation}
and hence
$$
\limsup_{t\to 0^+}4t\,\log u(x,t)\le -\frac{4 \sqrt{p'} d_\Ga(x)}{\la}+\frac4{\la^2}.
$$
If we choose $\la^*>0$ such that 
$$
-\frac{4 \sqrt{p'} d_\Ga(x)}{\la^*}+\frac4{(\la^*)^2}=-p'\,d_\Ga(x)^2, \ \mbox{ that is } \ \la^*=\frac2{\sqrt{p'}\, d_\Ga(x)},
$$
we obtain:
$$
\limsup_{t\to 0^+}4t\,\log u(x,t)\le -p'\,d_\Ga(x)^2.
$$
This inequality and \eqref{liminf ball} imply that \eqref{limit} holds pointwise in $\Om$. Since, for $x\in\Ga$, \eqref{limit} is trivial, this also holds pointwise in $\ol\Om$.
\end{proof}

\subsection{Constructing barriers for problem \eqref{G-heat}-\eqref{boundary}}

In this subsection, we shall prove our first-order asymptotics for quite general (not necessarily bounded) domains. A useful tool will be the function defined by 
\begin{equation}
\label{global solution}
\Phi(x,t)=t^{-\frac{N+p-2}{2(p-1)}}e^{-p'\frac{|x|^2}{4t}}, \quad (x,t)\in \RR^N\times(0,\infty),
\end{equation}
which is a solution of \eqref{G-heat} that generalizes to the case $p\not=2$ the fundamental solution of the heat equation, as shown in \cite[Proposition 4.1]{BG-IUMJ}. 

The next two lemmas shall give two global barriers for the solution of \eqref{G-heat}-\eqref{boundary}.

\begin{lem}[A barrier from below]
\label{below-barriers}
Let $z\in\RR^N\setminus\ol\Om$ and $U^z$ be the function defined by
\begin{equation}
\label{below-barriers formula}
U^z(x,t)=A_{N,p}\,d_\Ga(z)^\frac{N+p-2}{p-1}\Phi(x-z,t) \ \mbox{ for } \ (x,t)\in\ol{\Om}\times(0,\infty),
\end{equation}
with
$$
A_{N,p}=\left[\frac{p\, e}{2(N+p-2)}\right]^\frac{N+p-2}{2(p-1)}.
$$
\par
Assume that $u(x,t)$ is the bounded (viscosity) solution of \eqref{G-heat}-\eqref{boundary}, then we have that
$$
U^z\leq u\ \mbox{ on }\ \ol\Om\times(0,\infty).
$$
\end{lem}

\begin{proof}
If we consider 
$$
A(z)=\max\{\Phi(x-z,t):(x,t)\in\Ga\times(0,\infty)\},
$$
then the function $U(x,t)=A(z)^{-1}\Phi(x-z,t)$ satisfies \eqref{G-heat}, \eqref{initial} and is such that
$U\le 1=u$ on $\Ga\times(0,\infty)$.
Notice that, the function $U^z$ is decreasing with respect to $|x-z|$, then it is bounded by $1$ in the whole $\Om\times(0,\infty)$. Therefore, by applying Theorem \ref{parabolic cp} we have that
$
u\ge U\ \mbox{ on }\ \ol\Om\times(0,\infty).
$
\par
Since we directly compute that
$$
A(z)^{-1}=A_{N,p}\,d_\Ga(z)^\frac{N+p-2}{p-1},
$$
we get our claim.
\end{proof}

\begin{lem}[A barrier from above]
\label{above-barrier}
Let $u$ be the bounded (viscosity) solution of \eqref{G-heat}-\eqref{boundary} and $V=V(x,t)$ be the function defined by
\begin{equation}
\label{above-barrier formula}
V(x,t)=u_B(0,t/d_\Ga(x)^2)\ \mbox{ for }\ (x,t)\in\ol{\Om}\times(0,\infty),
\end{equation}
where $u_B$ is the solution of  \eqref{G-heat}-\eqref{boundary} in the unit ball $B$ and we mean that  $u_B(0,t/d_\Ga(x)^2)=1$ when $x\in\Ga$.
\par  
Then, it holds that
$
u\leq V 
$
on $\ol\Om\times(0,\infty)$.
\end{lem}

\begin{proof}
For $x\in\Ga$, the inequality is satisfied as an equality, by definition. Let $x\in\Om$ and let $v^x=v^x(y,t)$ be the solution of \eqref{G-heat}-\eqref{boundary} in $B^x\times(0,\infty)$, where $B^x$ is the ball centered at $x$ with radius $d_\Ga(x)$. The maximum principle and Theorem \ref{parabolic cp} give that
$$
u(y,t)\leq v^x(y,t)\ \mbox{ for every } \ (y,t)\in\ol{B^x}\times(0,\infty),
$$ 
and hence, in particular, $u(x,t)\leq v^x(x,t)$ for every $t>0$. Since $x$ is arbitrary in $\Om$, we infer that
\begin{equation}
\label{inequal}
u(x,t)\le v^x(x,t)\ \mbox{ for }\ (x,t)\in\Om\times(0,\infty).
\end{equation}
\par
Now, for fixed $x\in \Om$, consider the function defined by 
$$
w(y,t)=v^x(x+d_\Ga(x)\,y,d_\Ga(x)^2 t)\ \mbox{ for }\ (y,t)\in \ol{B}\times(0,\infty);
$$
since \eqref{G-heat} is translation and scaling invariant, we have that $w$ satisfies the problem \eqref{G-heat}-\eqref{boundary} in $B$, and hence equals $u_B$ on $\ol{B}\times(0,\infty)$. 
\par
Therefore, evaluating $u_B$ for $(0, t/d_\Ga(x)^2) $ gives that
$$
u_B(0,t/d_\Ga(x)^2)=w(0,t/d_\Ga(x)^2)=v^x(x,t)\ge u(x,t),
$$
by \eqref{inequal}, and this concludes the proof.
\end{proof} 

\subsection{First-order asymptotics for general domains}
\label{sec:first-order-general-domains}
We are now ready to prove our short-time asymptotic result for the solution of \eqref{G-heat}-\eqref{boundary}.

\begin{thm}[Pointwise convergence]
\label{th:pointwise}
Set $1<p\le \infty$. Let $\Om$ be a domain in $\RR^N$, with boundary $\Ga$ such that $\Ga=\pa(\RR^N\setminus\ol\Om)$, and let $u$ be the bounded (viscosity) solution of \eqref{G-heat}-\eqref{boundary}.
\par 
Then, we have that
\begin{equation}
\label{pointwise-limit}
\lim_{t\to 0^+}4t\log \left[u(x,t)\right]=-p'\,d_\Ga(x)^2 \ \mbox{ for every } x\in\ol{\Om}.
\end{equation}
\end{thm}

\begin{proof}
It is clear that \eqref{pointwise-limit} holds for $x\in\Ga$. Thus, for $x\in\Om$, set $r=d_\Ga(x)$ and
let $z$ be a point in $\Ga\cap\pa B_r(x)$. Since $\Ga=\pa(\RR^N\setminus\ol\Om)$, there is a sequence of points in $z_n\in\RR^N\setminus\ol{\Om}$ that converges to $z$.
\par
We first compute the limit in \eqref{pointwise-limit}, by replacing $u$ by the barriers constructed in Lemmas \ref{below-barriers} and \ref{above-barrier}.
In fact, from \eqref{below-barriers formula} and \eqref{global solution}, we easily compute that
$$
\lim_{t\to 0^+} 4t\,\log U^{z_n}(x,t)=-p' |x-z_n|^2,
$$
whereas, by Theorem \ref{Asymptotic in the ball}, we infer that 
\begin{equation}
\label{convergence-from-above}
\lim_{t\to 0^+} 4t\,\log V(x,t)=-p'\,d_\Ga(x)^2.
\end{equation}
\par
Now, for each $n\in\NN$, Lemmas \ref{below-barriers} and \ref{above-barrier} tell us that
\begin{equation*}
4t\log U^{z_n}(x,t)\leq 4t\log u(x,t)\leq 4t\log V(x,t),
\end{equation*}
for $t>0$. Thus, we get that
\begin{multline*}
-p' |x-z_n|^2=\liminf_{t\to 0^+}4t\log U^z(x,t)\le \liminf_{t\to 0^+} 4t\log u(x,t)\le \\
\le \limsup_{t\to 0^+}4t\log u(x,t)\leq \limsup_{t\to 0^+}4t \log V(x,t)=-p'\,d_\Ga(x)^2.
\end{multline*}
Letting $z_n$ tend to $z$ gives the conclusion, since $|x-z|=d_\Ga(x)$.
\end{proof}

\begin{rem}
{\rm
\label{rem:uniform-convergence-V}
Notice that the convergence in \eqref{convergence-from-above} is uniform on every subset of $\ol\Om$ where the $d_\Ga$ is bounded; actually, we obtain that
$$
4t\log V(x,t)+p'\,d_\Ga(x)^2=O(t\log t),
$$
for $t\to 0^+$, on such subsets.  Indeed, by a comparison with the one-dimensional solution as done in Theorem \ref{Asymptotic in the ball}, we can write that
\begin{equation}
\label{from above}
4t\, \log \Erfc\left(\frac{\sqrt{p'} d_\Ga(x)}{2 \sqrt{t}}\right)\le 4t\, \log V(x,t)
\end{equation}
while \eqref{lambda-inequality} with $\la=\frac{2}{\sqrt{p'}}$, \eqref{ball solution formula} and some manipulations give for $1<p<\infty$ that
\begin{multline}
  \label{eq:uniform-from-above}
  4t\log V(x,t)\leq\\
-p'\,d_\Ga(x)^2+4t\log\left[\frac{\int_{0}^{\pi}(\sin\te)^\al\,d\te}{\int_{0}^{\pi}e^{-p'\frac{1-\cos\te}{2t}d_\Ga(x)^2}(\sin\te)^\al\,d\te}\right].
\end{multline}
The explicit expressions in \eqref{from above} and \eqref{eq:uniform-from-above}, imply the desired claim. 
\par 
The case $p=\infty$ is analogous, simpler, and even yields the better behavior $O(t)$ as $t\to 0$.
}
\end{rem}

\par
With some extra sufficient condition on the regularity of $\Om$, we can obtain uniform convergence.
\par
Let $\om:(0,\infty)\to (0,\infty)$ be a strictly increasing continuous function such that $\om(\tau)\to 0$ as $\tau\to 0^+$. We say that a domain $\Om$ is of class $C^{0,\om}$, if there exists a number $r>0$ such that, for every point $x\in\Ga$, there is a coordinate system $(y',y_N)\in\RR^{N-1}\times\RR$, and a function $\phi:\RR^{N-1}\to\RR$ such that
\begin{enumerate}[(i)]
\item
$B_r(x)\cap\Om=\{(y',y_N)\in B_r(x):y_N<\phi(y')\}$;
\item
$B_r(x)\cap\Ga=\{(y',y_N)\in B_r(x):y_N=\phi(y')\}$;
\item
$|\phi(y')-\phi(z')|\le\om(|y'-z'|)$ for all $(y',\phi(y')), (z',\phi(z'))\in B_r(x)\cap\Ga$.
\end{enumerate}

\begin{thm}[Uniform convergence]
\label{th:uniform}
Let $1<p\le\infty$, suppose that $\Om$ is a domain of class $C^{0,\om}$, and set
$$
\psi(t)=\min_{0\le s\le R}\sqrt{s^2+[t-\om(s)]^2}.
$$
Let $u$ be the bounded (viscosity) solution of \eqref{G-heat}-\eqref{boundary}.
\par
Then, it holds that
\begin{equation}
\label{uniform-estimate}
4t\,\log u(x,t)+p'\,d_\Ga(x)^2=O(t \log \psi(t)) \ \mbox{  as } \ t\to 0^+,
\end{equation}
uniformly on every compact subset of $\ol\Om$. 
In particular, if $t \log\psi(t)\to 0$ as $t\to 0^+$,
then the solution $u$ of \eqref{G-heat}-\eqref{boundary} satisfies \eqref{pointwise-limit} uniformly on every compact subset of $\ol\Om$.
\end{thm}

\begin{proof}
In Remark \ref{rem:uniform-convergence-V}, we already observed that in \eqref{pointwise-limit} the convergence is uniform from above. 
We shall now modify the argument for the barrier $U^{z_n}(x,t)$ in such a way that it becomes uniform in $x$. In fact, we will choose the points $z_n$ along a suitable curve parametrized upon the time $t$. 
\par
For every $x\in\Om$, we choose a coordinate system $(y',y_N)\in\RR^{N-1}\times\RR$, with its origin at a point in $\Ga$ at minimal distance $d_\Ga(x)$ from $x$. In this coordinate system, we choose $z(t)=(0',t)$ that, if $t$ is small enough is by construction a point in $\RR^N\setminus\ol{\Om}$, since $t>\phi(0')$. Also, by our assumptions on $\Om$, $d_\Ga(z(t))$ is bounded from below by the distance of $z(t)$ from the graph of the function $y'\mapsto\om(|y'|)$ defined for $y'\in\{y\in B_r(0): y_N=0\}$, that is
$$
d_\Ga(z(t))\ge \min_{0\le s\le r}\sqrt{s^2+[\om(s)-t]^2}.
$$ 
\par
It is clear that this construction does not depend on the particular point $x\in\ol{\Om}$ chosen, but only on the regularity assumptions on $\Om$. Thus, we define our uniform barrier from below by $U^{z(t)}(x,t)$ and hence, from the definition of $U^z(x,t)$, we have that
\begin{multline*}
4t\,\log U^{z(t)}(x,t)=\\
4t\,\log\left[A_{N,p}\,t^{-\frac{N+p-2}{2(p-1)}}\right]+4t\,\frac{N+p-2}{p-1}\,\log d_\Ga(z(t))-p'\,|z(t)-x|^2,
\end{multline*}
and hence
\begin{multline}
\label{from below}
4t\,\log U^{z(t)}(x,t)\ge \\
4t\,\log\left[A_{N,p}\,t^{-\frac{N+p-2}{2(p-1)}}\right]+4t\,\frac{N+p-2}{p-1}\,\log\psi(t)-p'\,[d_\Ga(x)+t]^2,
\end{multline}
since $|z(t)-x|\le|z(t)|+|x|=t+d_\Ga(x)$.
The desired estimate \eqref{uniform-estimate} follows from an inspection of \eqref{from below} and Remark \ref{rem:uniform-convergence-V}.
\end{proof}

\begin{rem}
{\rm
Under sufficient assumptions on $\om$, we can replace $\psi$ by $a\,\om^{-1}$, for some positive constant $a$, where $\om^{-1}$ is the inverse function of $\om$. For instance, if $\Om$ is of class $C^\al$, with $0<\al<1$ --- that means that $\Ga$ is locally a graph of an $\al$-H\"older continuous function --- then the assumptions of Theorem \ref{th:uniform} are fulfilled, since $\psi(t)\ge a\,t^{1/\al}$ as $t\to 0^+$.
}
\end{rem}

The same assertion of Theorem \ref{th:uniform} holds true even if we replace $1$ in  \eqref{boundary} by a bounded time-dependent non-constant  boundary data, provided that this is bounded away from zero.

\begin{cor}
\label{positive  boundary data}
Let $w$ be the bounded solution of \eqref{G-heat}, \eqref{initial} satisfying
$$
w=h\ \mbox{ on }\ \Ga\times(0,\infty),
$$
where the function $h:\Ga\times(0,\infty)\to\RR$ is such that 
$$
\ul{h}\leq h \leq \ol{h} \ \mbox{ on } \ \Ga\times(0,\infty),
$$ 
for some positive numbers $\ul{h}, \ol{h}$.
\par 
Then, we have that
$$
4 t \log w(x,t)= -p'd_\Ga(x)^2+O(t\log \psi(t)) \ \mbox{ as }\ t\to 0^+,
$$
uniformly on every compact subset of $\ol\Om$.
\end{cor}

\begin{proof}
Since $\ul{h}\, u \leq w \leq \ol{h}\, u$ on $\Ga\times(0,\infty)$, we can apply Theorem  \ref{parabolic cp} to get:
$$
\ul{h}\, u(x,t)\leq w(x,t) \leq \ol{h}\, u(x,t) \ \mbox{ on } \ \ol\Om\times(0,\infty).
$$
This implies that, for every $x\in\ol\Om$ and $t>0$,
\begin{equation*}
  4t \log\ul{h}+4 t \log u(x,t) \leq 4 t \log w(x,t) \leq  4t \log\ol{h}+4 t \log u(x,t).
\end{equation*}
The conclusion then easily follows from Theorem \ref{th:uniform}.
\end{proof}

The next corollary of  Theorem \ref{th:uniform} will be useful in Section \ref{sec:second-parabolic}.

\begin{cor}
\label{cor:uniform}
Let $v:\ol{\Om}\times(0,\infty)\to\RR$ be defined by 
$$
\Erfc\left(\frac{\sqrt{p'}\, v(x,t)}{2 \sqrt{t}}\right)=u(x,t) \ \mbox{ for } \ (x,t)\in\ol{\Om}\times(0,\infty).
$$
Then
$$
v(x,t)=d_\Ga(x)+O(t\log \psi(t)) \ \mbox{ as } \ t\to 0^+,
$$
uniformly on every compact subset of $\ol\Om$.
\end{cor}

\begin{proof}
From the definition of $v(x,t)$, operating as in the proof of Proposition \ref{Asymp hspace} yields that
$$
4t\,\log u(x,t)+p'\,v(x,t)^2=4t\,\log\left( \sqrt{\frac{p'}{4\pi}}\int_0^\infty e^{-\frac12 p' \frac{v(x,t)}{\sqrt{t}}\,\si-\frac14 p' \si^2}d\si \right)\le 0.
$$
\par
By this inequality, since the first summand at the left-hand side converges uniformly on every compact $K\subset\ol{\Om}$ as $t\to 0^+$, we can infer that there exist $\ol{t}>0$ and $\de>0$ such that $0\le v(x,t)\le\de$ for any $x\in K$ and $0<t<\ol{t}$. 
\par
Thus, for $x\in K$ we have that 
\begin{multline*}
-\left[4t\,\log u(x,t)+p'\,d_\Ga(x)^2\right]+4t\,\log\left( \sqrt{\frac{p'}{\pi}}\int_0^\infty e^{-\frac12 p' \frac{\de}{\sqrt{t}}\,\si-\frac14 p' \si^2}d\si \right)\le \\
p'\,\left[v(x,t)^2-d_\Ga(x)^2\right]\le -\left[4t\,\log u(x,t)+p'\,d_\Ga(x)^2\right],
\end{multline*}
which implies the desired uniform estimate, by means of \eqref{uniform-estimate}.
\end{proof}

\section{Second-order asymptotics}
\label{sec:second-parabolic}

In this section, we shall suppose that $\Om$ is a domain of class $C^2$ (not necessarily bounded) and, for any point $y\in \Ga$, denote by $\ka_1(y),\dots,\ka_{N-1}(y)$ the principal curvatures of $\Ga$ at $y$ with respect to the interior normal direction to $\Ga$. Moreover, we let $\Pi_\Ga$ be the function defined in \eqref{def-function-Pi}:
$$
\Pi_\Ga(y)=\prod_{j=1}^{N-1}\left[1-R\,\ka_j(y)\right] \ \mbox{ for } \ y\in\Ga.
$$
We then recall a useful geometrical lemma (\cite[Lemma 2.1]{MS-PRSE}. 

\begin{lem}
\label{lem:geometric-asymptotics}
Let $x\in\Om$ and assume that, for $R>0$, there exists $y_x\in\Ga$ such that $\ol{B_R(x)}\cap(\RR^N\setminus\Om)=\{y_x\}$ and that $\ka_j(y_x)<1/R$ for $j=1,\dots,N-1$.
\par 
Then, it holds that
$$
\lim_{s\to 0^+}s^{-\frac{N-1}{2}}\cH_{N-1}(\Ga_s\cap B_R(x))=\frac{\om_{N-1}\,(2R)^{\frac{N-1}{2}}}{(N-1)\sqrt{\Pi_\Ga(y_x)}}, 
$$
where $\cH_{N-1}$ denotes $(N-1)$-dimensional Hausdorff measure and $\om_{N-1}$ is the surface area of a unit sphere in $\RR^{N-1}$.
\end{lem}

\subsection{Short-time asymptotics for heat content}
\label{sec:asymptotics-heat-content}
Our asymptotic result for the heat content of a ball $B_R(x)$ is based on the following lemma.

\begin{lem}[Short-time asymptotics for a barrier]
\label{lem:barrier-asimptotics}
Let $x\in\Om$ and assume that, for $R>0$, there exists $y_x\in\Ga$ such that $\ol{B_R(x)}\cap(\RR^N\setminus\Om)=\{y_x\}$ and that $\ka_j(y_x)<1/R$ for $j=1,\dots,N-1$.
\par
Let $f$ be a continuous function on $\RR$ such that 
$$
\lim_{s\to\infty} f(s)=0 \quad \mbox{ and } \quad \int_0^\infty s^\frac{N-1}{2}f(s)\,ds<\infty.
$$ 
Let $\xi, \eta:(0,\infty)\to(0,\infty)$ be two functions of time such that 
$\xi(t)$ is positive in $(0,\infty)$, and
$$
\lim_{t\to 0^+} \xi(t)=\lim_{t\to 0^+} \eta(t)=0.
$$ 
\par
Then it holds that
\begin{multline}
\label{general-asymptotics}
\lim_{t\to 0^+} \xi(t)^{-\frac{N+1}{2}}\int_{B_R(x)}f\left(\frac{d_\Ga(z)}{\xi(t)}+\eta(t)\right) dz=\\
\frac{\om_{N-1}\,(2R)^\frac{N-1}{2}}{(N-1)\sqrt{\Pi_\Ga(y_x)}}\,\int_0^\infty s^\frac{N-1}{2}f(s)ds.
\end{multline}
\end{lem}

\begin{proof}
By the co-area formula and a simple change of variables, we have that
\begin{multline*}
\int_{B_R(x)}f\left(\frac{d_\Ga(z)}{\xi(t)}+\eta(t)\right)\,dz=\\
\int_0^{2R}f\left(\frac{s}{\xi(t)}+\eta(t)\right)\,\cH_{N-1}(\Ga_s\cap B_R(x))\,ds=\\
\xi(t)\,\int_{\eta(t)}^{\frac{2R}{\xi(t)}+\eta(t)}f(\si)\,\cH_{N-1}(\Ga_{\xi(t)[\si-\eta(t)]}\cap B_R(x))\,d\si.
\end{multline*}
Thus,
\begin{multline*}
\xi(t)^{-\frac{N+1}{2}}\,\int_{B_R(x)}f\left(\frac{d_\Ga(z)}{\xi(t)}+\eta(t)\right)\,dz=\\
\int_{\eta(t)}^{\frac{2R}{\xi(t)}+\eta(t)}\left[\si-\eta(t)\right]^\frac{N-1}{2}\!\!f(\si)\,\frac{\cH_{N-1}(\Ga_{\xi(t) [\si-\eta(t)]}\cap B_R(x))}{\xi(t)^\frac{N-1}{2} [\si-\eta(t)]^\frac{N-1}{2}}\,d\si.
\end{multline*}
Therefore, by taking the limit as $t\to 0^+$, we obtain \eqref{general-asymptotics} by Lemma \ref{lem:geometric-asymptotics} and the dominated convergence theorem, after observing that there are constants $c, \ol{t}>0$ such that
$$
[\si-\eta(t)]^\frac{N-1}{2}\le c\,(\si^\frac{N-1}{2}+1) \ \mbox{ for } \ \eta(t)\le\si<\infty
$$
and $0<t<\ol{t}$.
\end{proof}

We are now ready to prove our formula for the heat content of $u$.

\begin{thm}[Short-time asymptotics for heat content]
\label{th:asymptotics-heat-content-p}
Set $1<p\leq\infty$. Let $x\in\Om$ and assume that, for $R>0$, there exists $y_x\in\Ga$ such that $\ol{B_R(x)}\cap(\RR^N\setminus\Om)=\{y_x\}$ and that $\ka_j(y_x)<1/R$ for $j=1,\dots,N-1$.
\par
If $u$ is the bounded (viscosity) solution of \eqref{G-heat}-\eqref{boundary}, 
then it holds that
\begin{equation}
\label{asymptotics to the curvatures}
\lim_{t\to 0^+} t^{-\frac{N+1}{4}}\int_{B_R(x)}u(z,t)\,dz=\frac{c_N}{(p')^{\frac{N+1}{4}}} \frac{R^\frac{N-1}{2}}{\sqrt{\Pi_\Ga(y_x)}},
\end{equation}
where
$$
c_N=\frac{2^{\frac{N+3}{2}}\pi^\frac{N-1}{2}}{(N+1)\Ga\left(\frac{N+1}{4}\right)}
$$
and $\Ga(\cdot)$ denotes Euler's gamma function.
\end{thm}

\begin{proof}
Let $x\in\Om$ and $B=B_R(x)$. We set
\begin{equation}
\label{def-eta}
\eta_B(t)=\frac1{\sqrt{t}}\,\max_{z\in\ol{B}}|v(z,t)-d_\Ga(z)| \ \mbox{ for } \ t>0.
\end{equation}
Since we are assuming that $\Ga$ is of class $C^2$, Corollary \ref{cor:uniform}  implies the estimate
$$
d_\Ga(z)-\sqrt{t}\,\eta_B(t)\le v(z,t)\le d_\Ga(z)+\sqrt{t}\,\eta_B(t), \ x\in\ol{B},
$$
with $\eta_B(t)=O(\sqrt{t}\,\log t)$ as $t\to 0^+$. Thus, we infer that
$$
\Erfc\left(\sqrt{\frac{p'}{4t}}\,d_\Ga(z)+\eta_B(t)\right)<u(z,t)<\Erfc\left(\sqrt{\frac{p'}{4t}}\,d_\Ga(z)-\eta_B(t)\right),
$$
for any $(z,t)\in\ol{B}\times(0,\infty)$.
We then choose $\xi(t)=2\sqrt{t}/\sqrt{p'}$ and $f(s)=\Erfc(s)$, and check that the assumptions of Lemma \ref{lem:barrier-asimptotics} are satified.
\par
Thus, we compute that
$$
\int_0^\infty s^\frac{N-1}{2} \Erfc(s)\,ds=\frac{N-1}{2\sqrt{\pi}(N+1)}\,\Ga\left(\frac{N-1}{4}\right),
$$
and hence formula \eqref{asymptotics to the curvatures} then follows from Lemma \ref{lem:barrier-asimptotics}, after some straightforward computation.
 \end{proof}

\subsection{Asymptotics for $q$-means}

Formula \eqref{asymptotics to the curvatures} can also be seen as an asymptotic formula for the {\it mean value} of $u$ on the ball $B_R(x)$. In fact, the following scale invariant formula follows:
$$
\lim_{t\to 0^+} \left(\frac{R^2}{t}\right)^\frac{N+1}{4}\dashint_{B_R(x)}u(y,t)\,dy=\frac{c_N}{(p')^\frac{N+1}{4}\sqrt{\Pi_\Ga(y_x)}},
$$
with
$$
c_N=\frac{2^\frac{N+1}{2}}{\sqrt{\pi}}\,\frac{N}{N+1}\,\frac{\Ga\left(\frac{N}{2}\right)}{\Ga\left(\frac{N+1}{4}\right)}.
$$
\par
Other statistical quantities that seem to be particularly appropriate and interesting in a game-theoretic  context are the so-called {\it $q$-means}. We shall consider for $1<q\le\infty$ the $q$-mean $\mu_q^u(x,t)$ of $u(\cdot,t)$ on $B_R(x)$, as defined in \eqref{p-mean}; this coincides with the mean value when $q=2$.

\begin{lem}[Asymptotics for the $q$-mean of a barrier]
\label{lem:asymptotics-p-mean}
Set $1<q<\infty$, let $x\in\Om$, and assume that, for $R>0$, there exists a point $y_x\in\Ga$ such that $\ol{B_R(x)}\cap(\RR^N\setminus\Om)=\{y_x\}$ and $\ka_j(y_x)<1/R$ for $j=1,\dots,N-1$.
\par
Let $\xi$ and $\eta$ be functions satisfying the assumptions of Lemma \ref{lem:barrier-asimptotics}. For a non-negative, decreasing and continuous function $f$ on $\RR$ such that
$$
\int_0^\infty f(\si)^{q-1} \si^\frac{N-1}{2} d\si<\infty,
$$
set
$$
w(y,t)=f\left(\frac{d_\Ga(y)}{\xi(t)}+\eta(t)\right) \ \mbox{ for } \ (y,t)\in\ol{\Om}\times(0,\infty).
$$
\par
If $\mu_q^w(x,t)$ is the $q$-mean of $w$ on $B_R(x)$, 
then the following formula holds:
\begin{equation}
\label{asymptotics-p-mean-w}
\lim_{t\to 0^+} \left(\frac{R}{\xi(t)}\right)^\frac{N+1}{2(q-1)}\!\!\!\!\mu_q^w(x,t)=
\left\{\frac{c_N\,\int_0^\infty f(\si)^{q-1} \si^\frac{N-1}{2} d\si}{\sqrt{\Pi_\Ga(y_x)}}\,\right\}^\frac1{q-1} 
\end{equation}
where
$
c_N=2^{-\frac{N+1}{2}} N!\,\Ga\left(\frac{N+1}{2}\right)^{-2}.
$
\end{lem}

\begin{proof}
We know from \cite{IMW} that $\mu(t)=\mu_q^w(x,t)$ is the unique root of the following equation
\begin{equation*}
\int_{B_R(x)}|w(y,t)-\mu(t)|^{q-2} [w(y,t)-\mu(t)]\, dy=0,
\end{equation*}
or
\begin{equation}
\label{characterization-mu}
\int_{B_R(x)}[w(y,t)-\mu(t)]_+^{q-1} dy=\int_{B_R(x)}[\mu(t)-w(y,t)]_+^{q-1} dy,
\end{equation}
where $[\si]_+=\max\{0,\si\}$.
\par
Firstly, we compute the short-time behavior of the left-hand side of \eqref{characterization-mu}. Let $\Ga_s=\{z\in B_R(x):d_\Ga(z)=s\}$. By the co-area formula, we get that
\begin{multline*}
  \int_{B_R(x)}[w(y,t)-\mu(t)]_+^{q-1} dy=\\
\int_{0}^{2R}\left[f\left(\frac{s}{\xi(t)}+\eta(t)\right)-\mu(t)\right]^{q-1}_+\cH_{N-1}\left(\Ga_s\right)\,ds.
\end{multline*}
\par
By the change of variable $s=\xi(t)\left[\si-\eta(t)\right]$, we obtain that
\begin{multline*}
  \int_{B_R(x)}[w(y,t)-\mu(t)]_+^{q-1} dy=\\
\xi(t)\int_{\eta(t)}^{\be(t)}\left[f\left(\si\right)-\mu(t)\right]^{q-1}_+\cH_{N-1}\left(\Ga_{\xi(t)\left[\si-\eta(t)\right]}\right)\,d\si,
\end{multline*}
where we set $\be(t)=\frac{2R}{\xi(t)}+\eta(t)$. 
\par 
Hence,
\begin{multline*}
  \xi(t)^{-\frac{N-1}{2}}\int_{B_R(x)}[w(y,t)-\mu(t)]_+^{q-1} dy=\\
\int_{\eta(t)}^{\be(t)}\frac{\cH_{N-1}\left(\Ga_{\xi(t)\left[\si-\eta(t)\right]}\right)}{\left\{\xi(t)\left[\si-\eta(t)\right]\right\}^{\frac{N-1}{2}}}\left[\si-\eta(t)\right]^{\frac{N-1}{2}}\left\{f(\si)-\mu(t)\right\}^{q-1}\,d\si.
\end{multline*}
Now, as $t\to 0^+$ we have that $\eta(t), \xi(t), \mu(t)\to 0$, $\be(t)\to \infty$ and that $\xi(t)\left[\si-\eta(t)\right]\to 0$ for almost every $\si\ge 0$. Thus, we can infer that
\begin{multline}
\label{eq:left-hand side}
  \lim_{t\to 0^+}
\xi(t)\int_{B_R(x)}[w(y,t)-\mu(t)]_+^{q-1} dy=\\
\frac{\om_{N-1}(2R)^{\frac{N-1}{2}}}{(N-1)\sqrt{\Pi_\Ga(y_x)}}\int_{0}^{\infty}f(\si)^{q-1}\si^{\frac{N-1}{2}}\,d\si,
\end{multline}
by Lemma \ref{lem:geometric-asymptotics} and an application of the dominated convergence theorem, as an inpection of the integrand function reveals.
\par 
Secondly, we treat the short-time behavior of the right-hand side of \eqref{characterization-mu}. By again performing the co-area formula and after some manipulations, we have that
\begin{multline}
\label{eq:right-hand side}
  \int_{B_R(x)}\left [\mu(t)-w(y,t)\right]^{q-1}_+=\\
\mu(t)^{q-1}\int_{0}^{2R}\left[1-f\left(\frac{s}{\xi(t)}+\eta(t)\right)\Big/\mu(t)\right]^{q-1}_+\cH_{N-1}\left(\Ga_s\right)\,ds
\end{multline}
which, on one hand,  leads to 
\begin{equation*}
    \int_{B_R(x)}\left [\mu(t)-w(y,t)\right]^{q-1}_+\le \mu(t)^{q-1}|B_R(x)|.
\end{equation*}
Notice in particular that, by using both \eqref{characterization-mu} and \eqref{eq:left-hand side}, the last inequality informs us that
$$
\mu(t)\ge c\,\xi(t)^{\frac{N+1}{2(q-1)}},
$$
for some positive constant $c$. Hence, after setting $\be(s,t)=\frac{s}{\xi(t)}+\eta(t)$, the assumptions on $f$ give the following chain of inequalities:
\begin{multline*}
  \int_{\be(s,t)/2}^{\infty}f(\si)^{q-1}\si^{\frac{N-1}{2}}\,d\si\ge\int_{\be(s,t)/2}^{\be(s,t)}f(\si)^{q-1}\si^{\frac{N-1}{2}}\,d\si\ge\\
\frac{2\left(1-2^{-\frac{N+1}{2}}\right)}{N+1}\frac{f\left(\be(s,t)\right)^{q-1}}{\xi(t)^{\frac{N+1}{2}}}\left[s+\eta(t)\,\xi(t)\right]^{\frac{N+1}{2}}\ge\\
\frac{2\left(1-2^{-\frac{N+1}{2}}\right)}{c(N+1)}\left[f\left(\frac{s}{\xi(t)}+\eta(t)\right)\Big/\mu(t)\right]^{q-1} \left[s+\eta(t)\,\xi(t)\right]^{\frac{N+1}{2}}.
\end{multline*}
Since, for almost every $s\ge 0$, the first term of the chain vanishes as $t\to 0^+$, we have that
\begin{equation*}
  \label{eq:pointwise-right-hand}
  \lim_{t\to 0^+}
\frac{f\left(\frac{s}{\xi(t)}+\eta(t)\right)}{\mu(t)}=0,
\end{equation*}
for almost every $s\ge 0$. Thus, \eqref{eq:right-hand side} gives at once that
\begin{equation}
\label{eq:right-hand side formula}
\lim_{t\to 0^+}   
\mu(t)^{1-q}\int_{B_R(x)}\left [\mu(t)-w(y,t)\right]^{q-1}_+=|B_R(x)|.
\end{equation}
\par 
Finally, \eqref{characterization-mu}, \eqref{eq:left-hand side} and \eqref{eq:right-hand side formula} tell us that
$$
\mu(t)^{q-1}=\xi(t)^\frac{N+1}{2} \frac{\om_{N-1}\,(2R)^{\frac{N-1}{2}}}{(N-1)\sqrt{\Pi_\Ga(y_x)}}\,\frac{\int_0^\infty f(\si)^{q-1}\,
\si^\frac{N-1}{2} d\si+o(1)}{|B_R(x)|+o(1)},
$$
that gives \eqref{asymptotics-p-mean-w}, after straightforward calculations involving Euler's gamma function.
\end{proof}

\begin{rem}
{\rm
If $q=\infty$, we know that
\begin{multline*}
\mu_\infty^w(x,t)=\frac12\,\left\{ \min_{\ol{B_R(x)}} w(\cdot,t)+\max_{\ol{B_R(x)}} w(\cdot,t)\right\}=\\
\frac12\,\left[f\left(\frac{\ol{d}}{\xi(t)}+\eta(t)\right)+f(\eta(t))\right],
\end{multline*}
where $\ol{d}$ is positive, being the maximum of $d_\Ga$ on $\ol{B_R(x)}$. Hence, it is easy to compute:
$$
\lim_{t\to 0^+} \mu_\infty^w(x,t)=\frac12\,f(0).
$$
Thus, formula \eqref{asymptotics-p-mean-w} does not extend continuously to the case $q=\infty$.
}
\end{rem}

\begin{thm}[Short-time asymptotics for $q$-means]
\label{th:asymptotics-p-mean}
Let $x\in\Om$, and assume that, for $R>0$, there exists a point $y_x\in\Ga$ such that $\ol{B_R(x)}\cap(\RR^N\setminus\Om)=\{y_x\}$ and $\ka_j(y_x)<1/R$ for $j=1,\dots,N-1$.
\par
Set $1<p\leq \infty$ and suppose that $u$ is the bounded (viscosity) solution of \eqref{G-heat}-\eqref{boundary} and, for $1<q\le\infty$, $\mu_q^u(x,t)$ is its $q$-mean on $B_R(x)$.
\par
Then, the following formulas hold:
\begin{equation}
\label{asymptotics-p-mean-u}
\lim_{t\to 0^+} \left(\frac{R^2}{t}\right)^\frac{N+1}{4(q-1)}\!\!\!\!\mu_q^u(x,t)=
\left\{\frac{c_N\,\int_0^\infty \Erfc(\si)^{q-1} \si^\frac{N-1}{2} d\si}{(p')^\frac{N+1}{4}\sqrt{\Pi_\Ga(y_x)}}\right\}^\frac1{q-1} ,
\end{equation}
where $c_N= N!\, \Ga\left(\frac{N+1}{2}\right)^{-2}$, and
$$
\lim_{t\to 0^+} \mu_\infty^u(x,t)=\frac12.
$$
\end{thm}

\begin{proof}
From \cite{IMW}  we know that the functional $u\mapsto \mu_q^u(x,t)$ is monotonically increasing, that is $\mu_q^u(x,t)\le\mu_q^w(x,t)$  if $u\le w$ almost everywhere in $B_R(x)$.
Thus, as done in the proof of Theorem \ref{th:asymptotics-heat-content-p}, the limit in \eqref{asymptotics-p-mean-u} will result from Lemma \ref{lem:asymptotics-p-mean}, where we choose:
$$
w(x,t)=\Erfc\left(\sqrt{\frac{p'}{4t}}\,d_\Ga(y)\pm\eta(t)\right),
$$
that is we choose $\xi(t)=\sqrt{4t/p'}$ and $\eta(t)$ is still given by \eqref{def-eta}.
Thus, \eqref{asymptotics-p-mean-u} will follow at once from \eqref{asymptotics-p-mean-w}, where
$f(\si)=\Erfc(\si)$.
\par
By the same argument, we also get the case $q=\infty$, since $f(0)=1$.
 \end{proof}

\begin{rem}{\rm
Notice that
$$
\left\{\int_0^\infty \Erfc(\si)^{q-1} \si^\frac{N-1}{2} d\si\right\}^\frac1{q-1}
$$
can be seen as the $(q-1)$-norm of $\Erfc$ in $(0,\infty)$ with respect to the weighed measure $\si^\frac{N+1}{2} d\si$.
}
\end{rem}

\subsection{Symmetry results}
\label{sec:symmetry}

In this subsection, we present two applications with geometric flavor of our asymptotic formulas \eqref{pointwise-limit} and \eqref{asymptotics-p-mean-u}.
\par
An immediate result is that, even in the case \eqref{G-heat}-\eqref{boundary}, a time-invariant level surface is parallel to $\Ga$.

\begin{thm}
\label{th:parallelilsm-invariant level set}
Let $\Om$ be a domain in $\RR^N$ satisfying $\Ga=\pa\left(\RR^N\setminus\ol\Om\right)$ and suppose that, for $1<p\le\infty$, $u$ is the solution of \eqref{G-heat}-\eqref{boundary}. 
\par 
If $\Si\subset\Om$ is a time-invariant level surface for $u$, then there exists $R>0$ such that 
\begin{equation}
\label{constant distance}
d_\Ga(x)=R\ \mbox{ for every }\ x\in\Si.
\end{equation}
\end{thm}

\begin{proof}
For any choice of $x_1$ and $x_2\in\Si$, we have that $u(x_1,t)=u(x_2,t)$ and hence $4t\,\log u(x_1,t)=4t\,\log u(x_2,t)$ for every $t>0$. By Theorem \ref{Asymptotic in the ball},
we infer that $d_\Ga(x_1)=d_\Ga(x_2)$ and hence we obtain our claim.
\end{proof}

As an application of \eqref{asymptotics-p-mean-u}, we give a characterization of spheres, based on $q$-means, in the spirit of
\cite[Theorem 1.2]{MS-PRSE}. This result is new, even for the case $p=2$.

\begin{thm}
\label{th:radiality-bounded domain}
Set $1<p\le\infty$ and let $\Om$ be a domain of class $C^2$ with bounded and connected boundary $\Ga$. 
Let $u$ be the bounded (viscosity) solution of \eqref{G-heat}-\eqref{boundary}. 
\par 
Suppose that $\Si$ is a $C^2$-regular surface in $\Om$, that is a parallel surface to $\Ga$ at distance $R>0$.\par
If, for some $1<q<\infty$ and every $t>0$, the function
$$
\Si\ni x\mapsto \mu_q^u (x,t)
$$
is constant, then $\Ga$ must be a sphere.
\end{thm}

\begin{proof}
Since $\Si$ is of class $C^2$ and is parallel to $\Ga$, for every $y\in\Ga$, there is a unique $x\in\Si$ at distance $R$ from $y$.  Thus, owing to Theorem \ref{th:asymptotics-p-mean}, we can infer that 
$$
\Pi_\Ga=\mbox{constant on $\Ga$.}
$$ 
Our claim then follows from a  variant of Alexandrov's Soap Bubble Theorem (see \cite{Al}),  \cite[Theorem 1.2]{MS-PRSE}, or \cite[Theorem 1.1]{MS-AM}.  
\end{proof}

\appendix
\section{Viscosity solutions}

For the reader's convenience, in this appendix, we collect the salient results on the viscosity solutions of \eqref{G-heat} and \eqref{G-elliptic} that we use in this paper. The main references to the existing literature that we use are the classical treatises \cite{CIL} and \cite{CGG}, the papers \cite{DaL} and \cite{GGIS}, concerning useful versions of comparison and maximum principles for nonlinear parabolic equations, in addition to the more recent and specific works \cite{AP}, \cite{APR}, \cite{BG-IUMJ}, \cite{BG-CPAA}, and \cite{Do} on $\De_p^G$.

Equation \eqref{G-heat} can be formally written as
$$
u_t=F(\na u, \na^2 u),
$$
where
$$
F(\xi,X)=\tr[A(\xi) X], \ \mbox{ with } \ A(\xi)=\frac1{p}\,I+(1-2/p)\,\frac{\xi\otimes \xi}{|\xi|^2}
$$
for $\xi\in\RR^N\setminus\{0\}$.
Here, $I$ and $X$ denote the $N\times N$ identity and  a real symmetric matrix. 
\par
Since $F$ has a bounded discontinuity at $\xi=0$, by following \cite{CIL}, we need to consider the so-called lower and upper semi-continuos relaxations $F_*$ and $F^*$ of $F$, respectively defined as
$$
F_*(\xi, X)=\lim_{r\to 0^+} \inf\{F(\eta,Y): 0<|\eta-\xi|, |Y-X|<r\}, 
$$
and $F^*=-(-F)_*$.
\par
If $\xi\neq 0$, $F=F_*=F^*$ while, if $\xi=0$, we have that
\begin{eqnarray*}
&&p\,F_*(0,X)=\tr(X)+\max(p-2,0)\,\la(X)+\min(p-2,0)\,\La(X),\\
&&p\,F^*(0,X)=\tr(X)+\min(p-2,0)\,\la(X)+\max(p-2,0)\,\La(X),
\end{eqnarray*}
where $\la(X)$ and $\La(X)$ are the maximum and minimum eigenvalue of $X$  (see \cite{Do}, \cite{AP}).
For $\xi\neq 0$, $F$ is a linear operator in the variable $X$ that is \emph{uniformly elliptic}, since
$$
\min(1/p',1/p)\,I\leq A(\xi)\leq\max(1/p',1/p)\,I.
$$

\medskip

\subsection{Parabolic case}

We say that $u\in C\left(\Om\times\left(0,\infty\right)\right)$ is a viscosity solution of \eqref{G-heat} if, for every $(x,t)\in\Om\times\left(0,\infty\right)$, both of the following requirements are fulfilled:
\begin{enumerate}[(i)]
\item for every $\fhi\in C^2\left(\Om\times\left(0,\infty\right)\right)$ such that $u-\fhi$ attains its maximum at $(x,t)$, then
\begin{equation*}
\fhi_t(x,t)-F^*\left(\na\fhi(x,t),\na^2\fhi(x,t)\right)\leq 0;
\end{equation*}

\item for every $\fhi\in C^2\left(\Om\times\left(0,\infty\right)\right)$ such that $u-\fhi$ attains its minimum at $(x,t)$, then
\begin{equation*}
\fhi_t(x,t)-F_*\left(\na\fhi(x,t),\na^2\fhi(x,t)\right)\geq 0.
\end{equation*}

\end{enumerate}   

A quite general comparison principle holds for equation \eqref{G-heat} (even for unbounded domains).

\begin{thm}[\cite{GGIS}, Theorem 2.1]
\label{parabolic cp}
Let $u$ and $v$ be a viscosity subsolution and supersolution for \eqref{G-heat}, that are bounded in $\Om\times(0,\infty)$. 
\par
If $u\leq v$ on $\Ga\times(0,\infty)$ and $\ol{\Om}\times\{0\}$, then
$$
u\leq v \ \mbox{ on } \ \ol{\Om}\times [0,\infty).
$$
\end{thm} 

\begin{rem}
\label{maximum principle}
{\rm
Let $u$ be the (bounded) viscosity solution of \eqref{G-heat}-\eqref{boundary} in $\Om\times (0,\infty)$.  Theorem \ref{parabolic cp} and an application of the strong maximum principle for $u$ (see \cite[Corollary 2.4]{DaL}), then implies that
$$
0< u< 1, \ \mbox{ for } \ (x,t)\in\Om\times(0,\infty).
$$
Indeed, the constant functions $v\equiv 0$ and $w\equiv 1$ on $\Om\times(0,\infty)$ are solutions of  \eqref{G-heat} and $v\leq u\leq w$ on $\pa \Om\times(0,\infty)$ and $\ol{\Om}\times\{0\}$.
}
\end{rem}

\begin{rem}
{\rm
\label{viscosity consistency}
Since $F$, $F^*$, and $F_*$ are monotonic in $X$ (in the sense of \cite{CIL}), we have that any smooth function $u$ that satisfies (i) and (ii), in the classical sense, is a viscosity solution of \eqref{G-heat}.
}
\end{rem}
\par 

\begin{rem}
\label{existence of solutions}
{\rm
The existence of a solution of \eqref{G-heat}-\eqref{boundary} can be obtained as follows. Suppose that $\Ga$ is smooth enough, say $C^2$, and let $S$ be the parabolic boundary of $\Om\times(0,\infty)$, 
that is
$$
S=(\Ga\times(0,\infty))\cup(\ol\Om\times\{0\}).
$$
For any $g\in C^0\left(S\right)$, there exists a unique viscosity solution $u$ of \eqref{G-heat}, satisfying 
$u=g$ on $S$, as shown in \cite[Theorem 1.9]{MPR-SJMA} or \cite[Theorem 2.6]{BG-CPAA}.
\par 
Now, for any $n\in\NN$, let $g_n$ be a continuous function on $S$ such that $0\le g_n\le 1$ and
\begin{eqnarray*}
&&g_n=1\ \mbox{ on }\ \{(x,t)\in S:t>1/n\},\\
&&g_n=0\ \mbox{ on }\ \{(x,0):d_\Ga(x)>1/n\}.
\end{eqnarray*}
The solution $u_n$ of \eqref{G-heat} satisfying $u_n=g_n$ on $S$ is such that
 $0\leq u_n\leq 1$ on $\ol\Om\times[0,\infty)$, by Theorem \ref{parabolic cp}.
Thus, by virtue of uniform H\"older estimates (see \cite[Lemma 2.11]{BG-CPAA}), up to a subsequence, $u_n$ converges (locally uniformly) in $\Om\times(0,\infty)$ to a solution $u$ of \eqref{G-heat}. Moreover, by the choice of $g_n$, we have that $u$ satisfies the initial-boundary conditions \eqref{initial} and \eqref{boundary}. 
\par
By a result in \cite{AP}, we also know that $u\in C_{loc}^{1+\tau,\frac{1+\tau}{2}}\left(\Om\times(0,\infty)\right)$, for some $\tau>0$.
}
\end{rem}

\subsection{Elliptic case}
In the (non-homogeneous) elliptic case, the definition of viscosity solution is analogous (see \cite{CIL}, \cite{APR}). In fact, we say that a continuous function $u$ in $\Om$ is a viscosity solution of  \eqref{G-elliptic}, if for every $x\in\Om$, both of the following properties are satisfied:

\begin{enumerate}
\item[(iii)] for every $\fhi\in C^2\left(\Om\right)$ such that $u-\fhi$ attains its maximum at $x$, then
\begin{equation*}
\label{sub-solution}
\fhi(x)-\ve^2 F^*\left(\na\fhi(x),\na^2\fhi(x)\right)\leq 0;
\end{equation*}
\item[(iv)] for every $\fhi\in C^2\left(\Om\right)$ such that $u-\fhi$ attains its minimum at $x$, then
\begin{equation*}
\label{sub-solution}
\fhi(x)-\ve^2 F_*\left(\na\fhi(x),\na^2\fhi(x)\right)\geq 0.
\end{equation*}
\end{enumerate} 

A weak comparison principle is available for \eqref{G-elliptic}.

\begin{lem}[\cite{APR}, Appendix D]
\label{elliptic cp}
Let $u$ and $v$ be a viscosity sub-solution and a viscosity super-solution for \eqref{G-elliptic}. 
\par
If $u\leq v$ on $\pa\Om$, then
$
u\leq v
$
on $\ol{\Om}$.
\end{lem} 

We conclude this appendix with the following useful lemma.

\begin{lem}[Extension lemma]
\label{regular viscosity solutions}
Let $\Om$ be an open set and let $u\in C^2(\Om)$. Suppose that 
\begin{enumerate}[(i)]
\item
$x_0\in\Om$ is the unique critical point for $u$ in $\Om$;
\item
 $u$ is solution of \eqref{G-elliptic} in $\Om\setminus \{x_0\}$.
\end{enumerate}
\par
Then $u$ is a solution of \eqref{G-elliptic} in $\Om$.
\end{lem}

\begin{proof}
We can always assume that $p\geq 2$ since,  otherwise, we can switch the roles of $\la$ and $\La$.
\par
Let us now take a sequence of points $y_n\in\Om\setminus\{x_0\}$ such that $y_n\to x_0$ as $n\to\infty$.  From our assumption on $u$, we have that
$$
u-\frac{\ve^2}{p}\,\left\{\tr(\na^2 u) +(p-2)\,\frac{\lan \na^2u\na u,\na u\ran}{|\na u|^2}\right\}=0
$$
at each $y_n$.
Thus, both of the following inequalities hold at $y_n$:
$$
u-\frac1{p}\ve^2\,\left\{\tr(\na^2 u) +(p-2)\,\La(\na^2 u)\right\}\leq 0,
$$
and
$$
u-\frac1{p}\ve^2\,\left\{\tr(\na^2 u) +(p-2)\,\la(\na^2 u)\right\}\geq 0.
$$
By the continuity of the functions $\la(\na^2 u)$ and $\La(\na^2 u)$ in $\Om$, we then infer that the last two inequalities also hold at $x_0$.
The claim follows, by virtue of Remark \ref{viscosity consistency}, applied to the elliptic case.
\end{proof}

\section*{Acknowledgements}
This paper was partially supported by the Gruppo Nazionale Analisi Matematica, Probabilit\`a e Applicazioni (GNAMPA) of the Istituto Nazionale di Alta Matematica (INdAM).


\begin{thebibliography}{PSSW}

\bibitem[AS]{AS} M.~Abramowitz, I.~A.~Stegun, Handbook of mathematical functions with formulas, graphs, and mathematical tables. National Bureau of Standards Applied Mathematics Series, 55, $1972$ (10th printing with corrections).

\bibitem[Al]{Al} A.~D.~Alexandrov, \emph{Uniqueness theorems for surfaces in the large V}, Vestnik, Leningrad Univ. 13, 19 (1958), 5--8. Amer. Math. Soc. Transl. Ser. 2, 21 (1962), 412--416.

\bibitem[AP]{AP} A.~Attouchi, M.~Parviainen, \emph{H\"older regularity for the gradient of the inhomogeneous parabolic normalized $p$-Laplacian}, to appear in Commun. Contemp. Math. (available online), preprint arxiv 1610.04987. 

\bibitem[APR]{APR} A.~Attouchi, M.~Parviainen, E.~Ruosteenoja, \emph{$C^{1,\al}$ regularity for the normalized p-Poisson problem}, J. Math. Pures Appl. (9) 108 (2017), 553--591.

\bibitem[BG1]{BG-IUMJ} A.~Banerjee, N.~Garofalo, \emph{Gradient bounds and monotonicity of the energy for some nonlinear singular diffusion equations}, Indiana Univ. Math. J. 62 (2013), 699--736.
	
\bibitem[BG2]{BG-CPAA} A.~Banerjee, N.~Garofalo, \emph{On the Dirichlet boundary value problem for the normalized $p$-Laplacian evolution}, Commun. Pure Appl. Anal. 14 (2015), 1--21. 

\bibitem[BMS]{BMS} L.~Brasco, R.~Magnanini, P.~Salani, \emph{The location of the hot spot in a grounded convex conductor}, Indiana Univ. Math. J. 60 (2011), no.2, 633--659. 
	
\bibitem[CGG]{CGG} Y.~G. ~Chen, Y.~Giga, S.~Goto, \emph{Uniqueness and existence of viscosity solutions of generalized mean curvature flow equations}, J. Differential Geom. 33 (1991), 749--786.
		
\bibitem[CMS]{CMS-JEMS} G.~Ciraolo, R.~Magnanini, S.~Sakaguchi,
\emph{Symmetry of minimizers with a level surface parallel to the boundary},  J. Eur. Math. Soc. (JEMS) 17 (2015), 2789--2804.

\bibitem[CIL]{CIL} M.~Crandall, H.~Ishii, P.~L.~Lions, \emph{User's guide to viscosity solutions of second order partial differential equations}, Bull. Amer. Math. Soc. 27 (1992), 1--67. 

\bibitem[DaL]{DaL} F.~Da Lio, \emph{Remark on the strong maximum principle for viscosity solutions to fully nonlinear parabolic equations}, Commun. Pure Appl. Anal. 3 (2004), no. 3, 395--415.

\bibitem[Do]{Do} K.~Does, \emph{An evolution equation involving the normalized p-Laplacian}, Commun. Pure Appl. Anal. 10 (2011), 361--396.	 

\bibitem[EI]{EI} L.~C.~Evans, H.~Ishii, \emph{A PDE approach to some asymptotics problem concerning random differential equations with small noise intensities}, Ann. Inst. Henri Poincar\'e Anal. Nonlin. 2 (1985), 1--20.

\bibitem[ES]{ES} L.~C.~Evans, J.~Spruck, \emph{Motion of level sets by mean curvature}, J. Differential Geom. 33 (1991), 635--681.

\bibitem[FW]{FW} M.~I.~Freidlin and A.~D.~Wentzell, \emph{Random Perturbations of Dynamical Systems}. Springer-Verlag, New York, 1984.

\bibitem[GGIS]{GGIS} Y.~Giga, S.~Goto, H.~Ishii, M.~H.~Sato, \emph{Comparison principle and convexity preserving properties for singular degenerate parabolic equations on unbounded domains}, Indiana Univ. Math. J. 40 (1991), 443-470.

\bibitem[IMW]{IMW} M.~Ishiwata, R.~Magnanini, H.~Wadade, \emph{A natural approach to the asymptotic mean value property for the p-Laplacian}, Calc. Var. Part. Diff. Eqs. 56 (2017), no.4, Art. 97, 22 pp.

\bibitem[JS]{JS} T.~Jin, L.~Silvestre, \emph{H\"older gradient estimates for parabolic homogeneous $p$-Laplacian equations}, Jour. Math. Pures Appl. 108 (2017), 63-87.

\bibitem[Ju]{Ju} P.~Juutinen, \emph{Decay estimates in the supremum norm for the solutions to a nonlinear evolution equation}, Proc. Roy. Soc. Edimburgh 144 (2014), 557--566.

\bibitem[JK]{JK} P.~Juutinen, B.~Kawohl, \emph{On the evolution governed by the infinity Laplacian}, Math. Ann. 335 (2006), 819--851.

\bibitem[KKK]{KKK} B.~Kawohl, S.~Kr\"omer, J.~Kurtz, \emph{Radial eigenfunctions for the game-theoretic $p$-laplacian on a ball}, Diff. Integ. Eqs. 27 (2014), no. 7-8, 659--670.

\bibitem[LL]{LL} E.~Lieb, M.~Loss, \emph{Analysis}. American Mathematical Society, Providence, RI, 2001 (second version)

\bibitem[MM]{MM-NA} R.~Magnanini, M.~Marini, \emph{The Matzoh Ball Soup problem: a complete characterization}, Nonlin. Anal. 131 (2016), 170--181.

\bibitem[MPeS]{MPeS} R.~Magnanini, D.~Peralta-Salas, S.~Sakaguchi, \emph{Stationary isothermic surfaces in Euclidean 3-space}, Math. Ann. 364 (2016), 97--124.

\bibitem[MPrS]{MPrS} R.~Magnanini, J.~Prajapat, S.~Sakaguchi, \emph{Stationary isothermic surfaces and uniformly dense domains}, Trans. Amer. Math. Soc. 385 (2006), 4821-4841.

\bibitem[MS1]{MS-AM} R.~Magnanini, S.~Sakaguchi, \emph{Matzoh ball soup: heat conductors with a stationary isothermic surface}, Ann. of Math. 156 (2002), 931--946.

\bibitem[MS2]{MS-AN} R.~Magnanini, S.~Sakaguchi, \emph{On heat conductors with a stationary hot spot}, Ann. Mat. Pura Appl. 183 (2004), 1--23.

\bibitem[MS3]{MS-PRSE} R.~Magnanini, S.~Sakaguchi, \emph{Interaction between degenerate diffusion and shape of domain}, Proc. Royal Soc. Edinburgh 137A (2007), 373--388.

\bibitem[MS4]{MS-IUMJ} R.~Magnanini, S.~Sakaguchi, \emph{Stationary isothermic surfaces for unbounded domains}, Indiana Univ. Math. Journ. 56 (2007), 2723--2738.

\bibitem[MS5]{MS-AIHP} R.~Magnanini, S.~Sakaguchi, \emph{Nonlinear diffusion with a bounded stationary level surface},  Ann. Inst. Henri Poincar\'e Anal. Nonlin. 27 (2010), 937--952.

\bibitem[MS6]{MS-JDE} R.~Magnanini, S.~Sakaguchi, \emph{Stationary isothermic surfaces and some characterizations of the hyperplane in the $N$-dimensional Euclidean space}, J. Differential Eqs. 248 (2010), 1112--1119.

\bibitem[MS7]{MS-JDE2} R.~Magnanini, S.~Sakaguchi, \emph{Interaction between nonlinear diffusion and geometry of domain}, J. Differential Eqs. 252 (2012), 236--257.

\bibitem[MS8]{MS-MMAS} R.~Magnanini, S.~Sakaguchi, \emph{Matzoh ball soup revisited: the boundary regularity issue}, Math. Meth. Appl. Sci. 36 (2013), 2023--2032.

\bibitem[MPR]{MPR-SJMA} J.~J.~Manfredi, M.~Parviainen, J.~D.~Rossi, \emph{An asymptotic mean value characterization for a class of nonlinear parabolic equations related to tug-of-war games}, SIAM J. Math. Anal. 42 (2010), 2058--2081.

\bibitem[PSSW]{PSSW}
Y.~Peres, O.~Schramm, S.~Sheffield, D.~B.~Wilson, \emph{Tug-of-war and the infinity Laplacian},
J. Amer. Math. Soc. 22 (2009), 167--210.

\bibitem[PS]{PS} Y.~Peres, S.~Sheffield, \emph{Tug-of-war with noise: a game-theoretic view of the p-Laplacian}, Duke Math. J. 145 (2008), 91--120.

\bibitem[Ru]{Ru} W.~R.~Rudin, Principles of Mathematical Analysis, Third edition. McGraw-Hill, New York, 1976.

\bibitem[Sa]{Sa} S.~Sakaguchi, \emph{Interaction between fast diffusion and geometry of domain}, Kodai Math. J. 37 (2014), 680-701.

\bibitem[Va]{Va} S.~R.~S.~Varadhan, \emph{On the behavior of the fundamental solution of the heat equation with variable coefficients}, Commun. Pure Appl. Math. 20 (1967), 431--455. 

\end{thebibliography}
\end{document}